\documentclass[12pt]{amsart} 
\usepackage{amscd}
\usepackage{amssymb}
\usepackage{a4wide}
\usepackage{amstext}
\usepackage{amsthm}
\usepackage{xcolor}
\usepackage{cite}
\usepackage[T1,T2A]{fontenc}
\usepackage[utf8]{inputenc}
\usepackage[american]{babel}
\usepackage{url}
\usepackage{amsfonts}
\usepackage{amssymb, amsthm}
\usepackage{amsmath}
\usepackage{mathtools}
\usepackage{needspace}
\usepackage[pdftex]{graphicx}
\usepackage{hyperref}
\usepackage{datetime}
\usepackage{epigraph}
\usepackage{verbatim}
\usepackage{mathtools}
\usepackage{xcolor}
\usepackage{enumerate}
\usepackage[american]{babel}
\numberwithin{equation}{section}

\newcommand{\N}{\mathbb{N}}
\newcommand{\R}{\mathbb{R}}

\newcommand{\D}{\mathbb{D}}
\newcommand{\Cm}{\mathbb{C}}
\newcommand{\intt}{\int\limits}
\newcommand{\summ}{\sum\limits}

\newcommand{\eps}{\varepsilon}

\renewcommand{\phi}{\varphi}

\newtheorem{Thm}{Theorem}[section]
\newtheorem{theorem}[Thm]{Theorem}

\newtheorem{lemma}[Thm]{Lemma}

\newtheorem{corollary}[Thm]{Corollary}
\newtheorem{remark}[Thm]{Remark}

\newtheorem{definition}{Definition}

\begin{document}
\sloppy

\title[Functionals with extrema at reproducing kernels]
{Functionals with extrema at reproducing kernels}
\author{Aleksei Kulikov}
\address{Department of Mathematical Sciences, Norwegian University of Science and Technology, NO-7491 Trondheim, Norway
\newline {\tt lyosha.kulikov@mail.ru}
}
\thanks{This work was supported  by Grant 275113 of the Research Council of Norway}
\begin{abstract} We show that certain monotone functionals on the Hardy spaces and  convex functionals on the Bergman spaces are maximized at the normalized reproducing kernels among the functions of norm $1$, thus proving the contractivity conjecture of Pavlovi\'c and of Brevig, Ortega-Cerd\`a, Seip and Zhao and the Wehrl-type entropy conjecture for the $SU(1,1)$ group of Lieb and Solovej, respectively.
\end{abstract}
\maketitle
\section{Introduction}
In this paper we will be working with several analytic function spaces in the unit disk $\D = \{ z\in \Cm : |z| < 1\}$. We begin by defining the appropriate Hardy and Bergman spaces.
\begin{definition}
For $0 < p < \infty$ we say that a function $f$ analytic in $\D$ belongs to the Hardy space $H^p$ if
$$||f||_{H^p}^p = \sup_{0 < r< 1} \frac{1}{2\pi}\intt_0^{2\pi} |f(re^{i\theta})|^p d\theta < \infty.$$
\end{definition}

To define the Bergman space we first introduce the M\"{o}bius invariant hyperbolic measure on the unit disk, which corresponds to the metric of constant negative curvature $-4\pi$. For $z = x + iy\in \D$ we define it as 
$$dm(z) = \frac{1}{(1-|z|^2)^2} \frac{dxdy}{\pi}.$$
We will also sometimes denote hyperbolic measure of the set $A$ by $|A|_H= m(A)$.
\begin{definition}
For $0 < p < \infty$ and $\alpha > 1$ we say that a function $f$ analytic in $\D$ belongs to the Bergman space $A^p_\alpha$ if
$$||f||_{A^p_\alpha}^p = \intt_\D (\alpha - 1)|f(z)|^p (1-|z|^2)^\alpha dm(z) < \infty.$$
\end{definition}

Note that  for the function $f(z) \equiv 1$ we have $||f||_{H^p} = ||f||_{A^q_\alpha} = 1$ for all admissible values of $p, q$ and $\alpha$. 

An important property of these spaces is that point evaluations are continuous in them. Specifically, for all analytic functions $f$ we have 
\begin{equation}\label{pointh}
|f(z)|^p(1-|z|^2)^\alpha \le ||f||_{A^p_\alpha}^p
\end{equation}
 and 
\begin{equation}\label{pointb}
|f(z)|^r(1-|z|^2) \le||f||_{H^r}^r.
\end{equation}
Moreover, since polynomials are dense in all of these spaces, for each fixed function $f$ the quantities on the left-hand sides of \eqref{pointh} and \eqref{pointb} tend uniformly to $0$ as $|z|\to 1$.

One of the  first questions about such spaces that one could ask is when one of them is contained in  another one. It turns out that the most interesting case to consider is when $\frac{p}{\alpha}$ is held constant, in which case we have
$$A^p_\alpha \subset A^q_\beta,\quad \frac{p}{\alpha} = \frac{q}{\beta} = r,\quad p < q$$
and $H^r$ is contained in all these spaces. Moreover, the $H^r$ norm can be evaluated as the limit of these Bergman norms in the sense that for $f\in H^r$ we have
$$||f||_{H^r} = \lim_{\alpha \to 1} ||f||_{A^{r\alpha}_\alpha}.$$
Thus it is sometimes reasonable to denote $H^r$ by $A^r_1$.

Recently, however, it was asked whether these embeddings are actually contractions, that is whether the norm $ ||f||_{A^{r\alpha}_\alpha}$ is decreasing in $\alpha$. In the case of Bergman spaces this question was asked by Lieb and Solovej \cite{9}. They showed that such contractivity implies their Wehrl-type entropy conjecture for the $SU(1,1)$ group. In the case of contractions from the Hardy space to the Bergman spaces it was asked by Pavlovi\'c in \cite{13} and by Brevig, Ortega-Cerd\`a, Seip and Zhao \cite{7} in relation to coefficient estimates for analytic functions. In this paper we confirm these conjectures and moreover we prove  more general results where we replace the function $t^r$ with a general convex or monotone function, respectively.
\begin{theorem}\label{Hardy}
Let $G:[0,\infty)\to \R$ be an increasing function. Then the maximum value of
\begin{equation}\label{hardyineq}
\intt_\D G(|f(z)|^p(1-|z|^2))dm(z)
\end{equation}
is attained for $f(z) \equiv 1$, subject to the condition that $f\in H^p$ and $||f||_{H^p}=1$.
\end{theorem}
\begin{theorem}\label{Bergman}
Let $G:[0,\infty)\to \R$ be a convex function. Then the maximum value of
\begin{equation}\label{bergmanineq}
\intt_\D G(|f(z)|^p(1-|z|^2)^\alpha)dm(z)
\end{equation}
is attained for $f(z) \equiv 1$, subject to the condition that $f\in A^p_\alpha$ and $||f||_{A^p_\alpha}=1$.
\end{theorem}
Applying these theorems to the convex and increasing function $G(t) = t^s, s > 1$ we get that all the embeddings above between Hardy and Bergman spaces are contractions, that is we get the following corollary
\begin{corollary}
For all $0 < p < q < \infty$ and $1 < \alpha < \beta < \infty$ with $\frac{p}{\alpha}=\frac{q}{\beta} = r$ for all functions $f$ analytic in $\D$ we have
$$||f||_{A^q_\beta} \le ||f||_{A^p_\alpha} \le ||f||_{H^r}$$
with equality for $f(z) \equiv c$ for $c\in \Cm$.
\end{corollary}

In fact, we are able to prove more general results than Theorems \ref{Hardy} and \ref{Bergman}. Specifically, in the case of Hardy spaces we are able to prove a sharp bound for the hyperbolic measure of the superlevel sets of the function $|f(z)|^p(1-|z|^2)$, thus verifying also Conjecture $2$ from \cite{7}, while in the case of Bergman spaces our proof allows us to consider some not necessarily convex functions, see Theorem \ref{hardythm} and Remark \ref{rem} respectively. 

It is important to mention that the M\"{o}bius group acts not only on the measure $m$ but on the spaces $A^p_\alpha$ as well. Specifically, given a function $f\in A^p_\alpha$ and $w\in \D$, the function 
$$g(z) = f\left(\frac{z-\bar{w}}{1-zw}\right)\frac{(1-|w|^2)^{\alpha/p}}{(1-zw)^{2\alpha/p}}$$
also belongs to the space $A^p_\alpha$ and moreover it has the same norm as $f$ and the same distribution of the function $|f(z)|^p(1-|z|^2)^\alpha$ with respect to the measure $m$. In particular, when $f(z) \equiv 1$ in this way we get $g(z) = \frac{(1-|w|^2)^{\alpha/p}}{(1-zw)^{2\alpha/p}}$ and these functions also necessarily give us the maximal value in \eqref{hardyineq} and \eqref{bergmanineq}. Note that when $p = 2$, the spaces $A^p_\alpha$ are Hilbert spaces and these functions turn out to be (normalized) reproducing kernels at the point $\bar{w}$. By analogy, we will call them reproducing kernels even if $p\ne 2$ (they are in fact reproducing kernels for the dual space).

Lastly, let us also mention that all our results also hold true for the Hardy and Bergman spaces in the upper half-plane, either by using a conformal mapping from the unit disk or by directly translating our methods. We chose to work in the unit disk, however, since it  allows us to simplify some calculations in the proof. 

 The structure of the paper is as follows. In Section 2 we prove a general monotonicity theorem for the hyperbolic measure of the  superlevel sets of analytic functions, which is an adaptation of the ingenious method from \cite{3} to the hyperbolic setting. This will also be the only Section where we use analyticity and hyperbolic geometry in an essential way. Then, in Sections 3 and 4 we deduce from it Theorems  \ref{Hardy} and \ref{Bergman}, respectively. Finally, in Section 5 we briefly discuss an application of Theorem  \ref{Hardy} to  coefficient estimates for analytic functions.
\section{monotonicity for the superlevel sets}
Let $f$ be a function  analytic in $\D$  such that $u(z) = |f(z)|^a (1-|z|^2)^b$ is bounded and goes to $0$ uniformly as $|z|\to 1$. Then the superlevel sets $A_t = \{z: u(z) > t\}$ for $t > 0$ are compactly embedded into $\D$ and thus have finite hyperbolic measure $\mu(t) = m(A_t)$. The goal of this section is to prove the following theorem which says that a certain quantity related to this measure is decreasing.
\begin{theorem}\label{monotone}
Let $f:\D \to \Cm$ be an analytic function such that the function $u(z) = |f(z)|^a (1-|z|^2)^b$ is bounded and $u(z)$ tends to $0$ uniformly as $|z|\to 1$. Then the function $g(t) = t^{1/b}(\mu(t) + 1)$ is decreasing on the interval $(0, t_0)$, where $t_0 = \max_{z\in\D} u(z)$.
\end{theorem}

The reason we consider this specific function $g$ is that for $f(z) \equiv 1$ the function $g$ turns out to be constant.

The proof of this theorem is mostly based on the methods developed in \cite{3}, translated from the euclidean to the hyperbolic setting. To this end, we introduce the hyperbolic length, associated with the measure $m$.
\begin{definition}
For a curve $\gamma \subset \D$ we define its hyperbolic length $|\gamma|_h$ as
$$|\gamma|_h = \intt_\gamma \frac{|dz|}{(1-|z|^2)\sqrt{\pi}}.$$
\end{definition}
\begin{proof}[Proof of Theorem \ref{monotone}]
We start from the formula
\begin{equation}\label{derivative}
-\mu'(t) = \int_{u = t}|\nabla u|^{-1}\frac{|dz|}{\pi (1-|z|^2)^2}
\end{equation}
along with the claim that $\{u = t\} = \partial A_t$ and that this set is a smooth curve for almost all $t\in (0, t_0)$. These assertions are by no means trivial, but the proof almost literally follows the proof of Lemma $3.2$ from \cite{3} and the discussion before it so we present here only the informal geometric reasoning for  formula \eqref{derivative}. Since for $z\in\partial A_t$ we have $u(z) = t$, $\nabla u$ is orthogonal to $\partial A_t$ and moreover it is pointing into $A_t$ since for $z\in A_t$ we have $u(z) > t$. Thus, when we decrease $t$ by a small number $\eps$ at each point of $\partial A_t$ the set $A_{t-\eps}$ is expanded in the direction orthogonal to $\partial A_t$ by the value about $\frac{\eps}{|\nabla u|}$. The factor $\frac{1}{\pi(1-|z|^2)^2}$ appears in \eqref{derivative} because we want to differentiate the hyperbolic measure of $A_t$ and not the euclidean one.

Following the approach from \cite{3}, our next step is to apply the Cauchy--Schwarz inequality to the hyperbolic length of $\partial A_t$:
\begin{equation}\label{CS}
|\partial A_t|_h^2 = \left(\int_{\partial A_t} \frac{|dz|}{\sqrt{\pi}(1-|z|^2)}\right)^2\le \left(\int_{\partial A_t} |\nabla u|^{-1} \frac{|dz|}{\pi (1-|z|^2)^2}\right)\left(\int_{\partial A_t} |\nabla u| |dz|\right).
\end{equation}
The first integral on the right-hand side is $-\mu'(t)$ so to proceed we have to analyze the second one. 

Let $\nu$ be the outward normal to $\partial A_t$. As explained above, $\nabla u$ is parallel to it but directed in the opposite direction. Thus we have $|\nabla u| = -\nabla u \cdot \nu$. Also, we note that since for $z\in \partial A_t$ we have $u(z) = t$, we have for $z\in \partial A_t$
$$\frac{|\nabla u|}{t} = \frac{|\nabla u|}{u} = -\frac{\nabla u\cdot \nu}{u} = -(\nabla \log u)\cdot \nu.$$

Now the second integral on the right-hand side of \eqref{CS} can be evaluated by Green's theorem:
\begin{equation}\label{by parts}
\int_{\partial A_t} |\nabla u||dz| = -t\int_{\partial A_t}(\nabla \log u)\cdot \nu = -t\int_{A_t} \Delta \log u(z) dxdy.
\end{equation}
Note that here it is important that $u(z)\ne 0$ for $z\in A_t$ so the function $\log u$ is well-defined on $A_t$ and $\partial A_t$. We have $\Delta \log u(z) = a\Delta \log |f(z)| + b\Delta \log (1-|z|^2)$. Since $f(z)\ne 0$ for $z\in A_t$, the first term is just $0$ while the second one is $-4b\frac{1}{(1-|z|^2)^2}$ which is proportional to the hyperbolic metric $m$.  Thus, the right-hand side of \eqref{by parts} is equal to $4\pi bt|A_t|_H$.

Combining everything, we get
$$-\mu'(t) \ge \frac{|\partial A_t|_h^2}{4bt|A_t|_H}.$$

Our next step is to use the isoperimetric inequality for the hyperbolic plane \cite{14} (see also \cite{15, 16}). In the case of curvature $-4\pi$ it says that for any set $A$ we have
$$|\partial A|_h^2 \ge 4\pi |A|_H + 4\pi|A|_H^2.$$

Using it with $A = A_t$ and recalling that $|A|_H = \mu(t)$ we get
\begin{equation}\label{diff}
-\mu'(t) \ge \frac{1+\mu(t)}{bt}.
\end{equation}
Note that here we used that $|A|_H > 0$ (that is, $t < t_0$) to avoid division by $0$.

Rewriting \eqref{diff} for $g(t) = t^{1/b}(\mu(t) + 1)$  we get $g'(t) = t^{1/b}(\frac{\mu(t) +1}{bt} + \mu'(t)) \le 0$, thus $g$ is decreasing as required. 
\end{proof}
We believe that the key reason why the above argument works is the proportionality  ${\Delta \log ||K_z||\sim m(z)}$, where $K_z$ are the reproducing kernels for our spaces. Note that the same kind of proportionality, in the Euclidean setting, played a crucial role in \cite{3}.

Note that for the function $f(z) \equiv 1$ everywhere in the proof above we have  equalities for all values of $a$ and $b$. Indeed, in the Cauchy--Schwarz inequality the functions are constant by radial symmetry, thus it is an equality, and the hyperbolic isoperimetric inequality is an equality exactly when the set is a hyperbolic disk. In fact, it is also an equality for all reproducing kernels, but verifying this for the Cauchy--Schwarz part is a rather long computation. This is also true in the upper half-plane. The reason we chose to work in the unit disk is that for $f(z)\equiv 1$ these equalities are easy to verify.
\section{Weak-type estimate for the Hardy spaces}
In this section we are going to prove the following bound for the measure of the superlevel sets of  functions from the Hardy spaces. Theorem \ref{Hardy} follows from it by a simple integration. In what follows we retain the notation from the previous section.
\begin{theorem}\label{hardythm}
Let $f\in H^p$ have norm $1$ and put $u(z) = |f(z)|^p(1-|z|^2)$. Then for all $t\in (0, \infty)$ we have
\begin{equation}\label{hardy bound}
\mu(t) \le \max\left(\frac{1}{t}-1,0\right).
\end{equation}
Equality in \eqref{hardy bound} holds for all $0 < t < \infty$ if $f(z) \equiv 1$.
\end{theorem}
Note that this theorem verifies Conjecture $2$ from \cite{7}.
\begin{proof}
Put $t_0 = \max_{z\in\D} u(z)$. By the pointwise bound \eqref{pointh} we have $t_0 \le 1$. In particular, for $t\ge t_0$ the bound holds trivially.

Assume that there exists some $0 < t_1 < t_0$ such that $\mu(t_1) > \frac{1}{t_1} - 1$. Then $\mu(t_1) = \frac{c}{t_1} - 1$ for some $c > 1$. We claim that in that case for all $0 < t < t_1$ we have $\mu(t) \ge \frac{c}{t}-1$. 

Indeed, applying the pointwise bound once again together with $u(z)\to 0$ as $|z|\to 1$, we see that Theorem \ref{monotone} can be applied to $f$ with $a = p, b = 1$, and we get that $g(t) = t(\mu(t) + 1)$ is decreasing. Since $g(t_1) = c$ we get $g(t) \ge c, 0 < t < t_1$, which corresponds to $\mu(t) \ge \frac{c}{t}-1$.

Next	 we are going to use the fact that $||f||_{A^{pr}_r}^{pr} \to ||f||_{H^p}^p = 1$ as $r\to 1$. Note that we can express the $A^{pr}_r$ norms via $\mu(t)$:
$$||f||_{A^{pr}_r}^{pr} = c_r \intt_0^{t_0}\mu(t)t^{r-1}dt,$$
where $c_r = r(r-1)$ is so that $c_r\intt_0^1 (\frac{1}{t}-1)t^{r-1}dt = 1$. The precise value of $c_r$ is not important for us except that $c_r\to 0$ as $r\to 1$ which corresponds to the fact that the norm in the Hardy space is supported on the circle $\partial \D$ and not on the whole disk $\D$. By the above bound we have
\begin{equation}\label{temp3}
||f||_{A^{pr}_r}^{pr} \ge c_r \intt_0^{t_1} (\frac{c}{t} - 1)t^{r-1}dt.
\end{equation}
We have
$$1 = c_r\intt_0^1 (\frac{1}{t}-1)t^{r-1}dt = c_r\intt_0^{t_1} (\frac{1}{t}-1)t^{r-1}dt + c_r\intt_{t_1}^1 (\frac{1}{t}-1)t^{r-1}dt = A(r)+B(r).$$

Since $c_r\to 0$ as $r\to 1$, we see that $B(r)\to 0$ as $r\to 1$ as well because the function we are integrating is bounded. Therefore, $A(r)\to 1$ as $r\to 1$. On the other hand, we have
$$\frac{\frac{c}{t}-1}{\frac{1}{t}-1} = c + \frac{c-1}{\frac{1}{t}-1}\ge c,$$
thus the right-hand side of \eqref{temp3} is at least $cA(r)$. Therefore $1 =\lim_{r\to 1}||f||_{A^{pr}_r}^{pr} \ge c\lim_{r\to 1}A(r)=c$ which is a contradiction. 
\end{proof}
\begin{remark}
By looking more closely at the above proof we can actually get the following formula for the $H^p$-norm:
\begin{equation}
||f||_{H^p}^p = \sup_{t: \mu(t) > 0} t(\mu(t)+1) = \lim_{t\to 0} t(\mu(t) + 1) = \lim_{t\to 0}t\mu(t),
\end{equation}
 which seems to be new.
\end{remark} 
\begin{proof}[Proof of Theorem \ref{Hardy}] If $\lim_{t\to 0^{+}}G(t) > 0$ then for all $f\in H^p$ with $||f||=1$ the integral in \eqref{hardyineq} is $+\infty$. Similarly if $\lim_{t\to 0^{+}}G(t)  < 0$ then the integral is always $-\infty$. Thus, we can assume that $\lim_{t\to 0^{+}}G(t)  = 0$. Then this integral can be expressed via $\mu(t)$ as 
\begin{equation}\label{dG}
\intt_0^\infty \mu(t)dG(t).
\end{equation}
Note that here we used that the function $\mu(t)$ is continuous, that is the sets $\{u(z) = t\}$ have zero measure. 

Since $G$ is increasing, measure $dG(t)$ is positive. Thus, by \eqref{hardy bound} this integral is at most
$$\intt_0^\infty \max\left(\frac{1}{t}-1, 0\right)dG(t),$$
which is the value of \eqref{hardyineq} for $f(z)\equiv 1$.
\end{proof}
\section{proof of Theorem \ref{Bergman}}
As in the proof of Theorem \ref{Hardy}, we begin by observing that if $\lim_{t\to 0^{+}}G(t)  \ne 0$, then the integral in \eqref{bergmanineq} is always $\pm \infty$. Thus, we restrict ourselves to the case $\lim_{t\to 0^{+}}G(t)  = 0$. In that case, the integral in \eqref{bergmanineq} can be rewritten as 
$$\int_0^\infty \mu(t) G'(t)dt,$$
where $\mu(t) = m(\{ z: u(z) > t\})$ and $u(z) = |f(z)|^p(1-|z|^2)^\alpha$. We wrote here $G'(t)dt$ instead of $dG(t)$ as in \eqref{dG} because $G$ is convex and hence its derivative is an actual function and not just a measure. We also assume that $||f||_{A^p_\alpha} = 1$, that is
$$\int_0^\infty \mu(t)dt = \frac{1}{\alpha-1}.$$
Applying Theorem \ref{monotone} to $f$ with $a = p, b = \alpha$ we get $\mu(t) = \frac{g(t)}{t^{1/\alpha}} - 1$ where $g$ is decreasing on $(0, t_0)$ with $t_0 = \max_{z\in \D}u(z)$. By the pointwise bound \eqref{pointb} we know that $t_0 \le 1$.

We are going to make a change of variables $x = t^{1-1/\alpha}$ in both of the above integrals. We get
\begin{equation}\label{temp5}
\intt_0^{x_0}(h(x) - x^{1/(\alpha-1)})dx = \frac{1}{\alpha},
\end{equation}
where $x_0 = t_0^{1-1/\alpha} \le 1$ and $h(x) = g(x^{\alpha/(\alpha - 1)})$ is a decreasing function. We want to maximize
\begin{equation}\label{temp6}
\intt_0^{x_0}(h(x) - x^{1/(\alpha-1)})s(x)dx,
\end{equation}
where $s(x) = G'(x^{\alpha/(\alpha - 1)})$ is an increasing function. From now on we will treat $h$ as a generic decreasing function, forgetting for a while that it came from the distribution of a holomorphic function, and we also fix $x_0$ for now. To proceed we need the following lemma.
\begin{lemma}\label{decreasing lemma}
Let $s(x), 0 \le x \le x_0$, be a function satisfying 
\begin{equation}\label{weaker}
\int_0^Xs(x)dx \le Xs(X)
\end{equation}
for all $0 \le X \le x_0$. Then among all decreasing functions $h:[0, x_0]\to \R$ with $\int_0^{x_0}h(x) = c$ the constant function $h(x) = \frac{c}{x_0}$ maximizes  $\int_0^{x_0}s(x)h(x)dx$.
\end{lemma}
Note that any increasing function $s$ trivially satisfies  condition \eqref{weaker}. However, even though there is a slightly simpler proof of this lemma in the case of increasing $s$, the proof of the general case is still short, and therefore we decided to prove lemma in full generality.
\begin{proof}For clarity of exposition, we will only prove the lemma in the case when the function $h$ is in $C^1([0, x_0])$. The proof in the general case can be obtained by the slight modification of the argument below.

Put $S(X) = \int_0^Xs(x)dx$ and $l(x) = -h'(x)\ge 0$ and integrate by parts in both integrals. We get
$$x_0h(x_0) + \int_0^{x_0}xl(x)dx = c,$$
and we want to maximize
\begin{equation}\label{temp4}
S(x_0)h(x_0) + \int_0^{x_0}S(x)l(x)dx.
\end{equation}
Note that $\left(\frac{S(x)}{x}\right)' = \frac{xs(x) -S(x)}{x^2}\ge 0$, thus $\frac{S(x)}{x}$ is increasing. Therefore we have
$$\int_0^{x_0}S(x)l(x)dx \le \int_0^{x_0}x\frac{S(x_0)}{x_0}l(x)dx = \frac{S(x_0)}{x_0}\int_0^{x_0}xl(x)dx.$$
Plugging this into \eqref{temp4} we get
$$S(x_0)h(x_0) + \int_0^{x_0}S(x)l(x)dx \le \frac{S(x_0)}{x_0}\left(x_0h(x_0) + \int_0^{x_0}xl(x)dx\right) = c\frac{S(x_0)}{x_0}.$$
For $l(x) = 0$ (that is, $h(x)$ being constant) we have equality here.
\end{proof}
Note that, for a fixed $x_0$, we know the value of $\int_0^{x_0}h(x)dx$ from \eqref{temp5}, while in \eqref{temp6} we want to maximize $\int_0^{x_0}h(x)s(x)dx$ plus some constant depending on $x_0$. Therefore, if $x_0$ is fixed, then we may assume that  $h(x)$ is constant and equal to
$$C(x_0) = \frac{1}{x_0}\left(\frac{1}{\alpha}+ \frac{\alpha-1}{\alpha} x_0^{\alpha/(\alpha-1)}\right),$$
while \eqref{temp6} is equal to
$$A(x_0) = C(x_0)S(x_0) - \int_0^{x_0}x^{1/(\alpha-1)}s(x)dx.$$

Our next goal is to show that $A(x_0)$ is an increasing function of $x_0$. To do so we first assume that $C(x_0)$ is decreasing which we will prove later. Taking the derivative of $A$ we get
$$A'(x_0) = C'(x_0)S(x_0) + C(x_0)s(x_0) - x_0^{1/(\alpha-1)}s(x_0).$$
Since $S(x_0) \le x_0s(x_0)$ and $C'(x_0) \le 0$, this quantity is at least
$$C'(x_0)x_0s(x_0) + C(x_0)s(x_0) - x_0^{1/(\alpha-1)}s(x_0) = s(x_0)\left(x_0C'(x_0) + C(x_0) - x_0^{1/(\alpha-1)}\right).$$
The expression in the brackets turns out to be exactly $0$ -- this can be either verified by a direct computation or, slightly informally, deduced from the fact that for $s(x) = 1$ (which corresponds to $G(x) = x$) we should always get the same value and all our inequalities turn into equalities.

It remains to show that $C(x_0)$ is a decreasing function. We first outline an informal argument, which this time is a bit harder to make rigorous, so later we will explicitly compute $C'(x_0)$ and verify that it is negative. The informal argument is as follows: if we treat $h$ as coming from the distribution of the analytic function and if for some $x_1 < x_2$ we would have $C(x_1) < C(x_2)$, then this would mean that all the superlevel sets of the function corresponding to $x_2$ have larger hyperbolic measure so they can not both have the same $A^p_\alpha$-norm. But since we already abandoned Bergman spaces by treating $h$ as a generic decreasing function this argument is not rigorous so we will just compute $C'(x_0)$:
$$C'(x_0) = \frac{x_0^{a/(a-1)}-1}{ax_0^2}.$$
Thus, $C'(x_0)$ is nonpositive for $0 < x_0 \le 1$, therefore $A(x_0) \le A(1)$ for $0 < x_0 \le 1$.

 It remains to note that for the case $f(z)\equiv 1$ we have $x_0 = 1$ and $h(x)$ is constant, thus the value of \eqref{bergmanineq} is exactly $A(1)$ and the Theorem is proved.
\qed
\begin{remark}\label{rem}
As can be seen from the proof, it is enough to assume that the function $s(x) = G'(x^{\alpha/(\alpha-1)})$ satisfies \eqref{weaker} for all $X>0$ which is strictly weaker than the assumption that $G$ be convex.
\end{remark}
\begin{remark}
Note that we actually proved that the value of \eqref{bergmanineq} is at most $A(x_0)$ where $x_0 = t_0^{1-1/\alpha},\, t_0 = \max_{z\in \D} u(z)$. In particular, if the function $G$ is strictly convex such as $G(t)=t^s, 1 < s < \infty$ then the function $A$ is strictly increasing and thus the reproducing kernels are the only maximizers of \eqref{bergmanineq}. 
\end{remark}
\section{Coefficient estimates for Hardy spaces}
An important special case of contraction from the Hardy space to a Bergman space is $H^p \subset A^2_{2/p}$ for $0 < p < 2$. It turns out that, given $f(z) = \sum_{n = 0}^\infty a_nz^n\in A^2_{2/p}$, we can express its norm as follows
$$||f||_{A^2_{2/p}}^2 = \summ_{n = 0}^\infty \frac{|a_n|^2 }{c_{2/p}(n)},\quad c_{2/p}(n) = \binom{n + 2/p -1}{n}.$$
Thus, for a function $f\in H^p$ we have
\begin{equation}\label{coeff}
\summ_{n = 0}^\infty \frac{|a_n|^2}{c_{2/p}(n)} \le ||f||_{H^p}^2.
\end{equation}
This can be seen as sharpening of the classical Hardy-Littlewood inequality \cite{1}. Before only partial results were obtained in the direction of this inequality, such as proving it in the case $\frac{2}{p}\in \N$ \cite{5} and proving it for the first few coefficients \cite{2, 10} or with a constant slightly worse than $1$ \cite{11}. Note that inequality \eqref{coeff} is sharp in every coefficient since it turns into an equality for all reproducing kernels.

By the inductive argument from \cite{4} inequality \ref{coeff} can be extended to the functions from the Hardy space on the multidimensional disk $\D^k$: for $f(z_1, \ldots , z_k) = \summ_{n_1, \ldots, n_k\ge 0} a_{n_1, n_2, \ldots , n_k}z_1^{n_1}\ldots z_k^{n_k}$ we have
\begin{equation}\label{coef}
\summ_{n_1, \ldots, n_k\ge 0} \frac{|a_{n_1, n_2, \ldots , n_k}|^2}{c_{2/p}(n_1)\ldots c_{2/p}(n_k)}\le ||f||_{H^p(\D^k)}.
\end{equation}

By passing to the limit and noting that $c_{2/p}(0) = 1$, we can extend this inequality to the analytic functions on the infinite-dimensional disk $\D^\infty$. These spaces in turn can be realized as  spaces of Dirichlet series in the right half-plane. In this setting \ref{coef} takes the following form.
\begin{theorem}
For a Dirichlet polynomial $f(z) = \sum_{n = 1}^N \frac{a_n}{n^s}$ and $0 < p < 2$ we have
$$\sum_{n = 1}^N \frac{|a_n|^2}{d_{2/p}(n)} \le ||f||_{\mathcal{H}^p}^2,$$
where $d_{2/p}(n)$ are the coefficients in $\zeta(s)^{2/p} = \sum_{n = 1}^\infty \frac{d_{2/p}(n)}{n^s}$ and 
$$||f||_{\mathcal{H}^p}^p = \lim_{T\to \infty} \frac{1}{2T}\int_{-T}^T |f(it)|^pdt.$$
\end{theorem}
Combining this theorem with the methods from \cite{12} we can slightly improve the bounds for the pseudomoments of the Riemann zeta function, though the improvement is only in the constant factor and not asymptotic. 
\subsection*{Acknowledgments} I would like to thank Kristian Seip for introducing me to the topic of contractive inequalities, Yurii Lyubarskii for telling me about the paper \cite{3}, Ivan Izmestiev for a discussion on isoperimetric inequalities and Fabio Nicola, Jan Philip Solovej and Paolo Tilli for helpful comments. 

\end{document}